\documentclass[a4paper,reqno,12pt]{amsart}
	\usepackage[latin1]{inputenc}
\usepackage[T1]{fontenc}
\usepackage[english]{babel}
\usepackage{appendix}
\usepackage[english]{minitoc}
\usepackage[nice]{nicefrac}
\usepackage{mathpazo,xcolor} \usepackage[pdftex]{hyperref}
 \newtheorem{thm}{Theorem}[section] \newtheorem{cor}[thm]{Corollary}  \newtheorem{lem}[thm]{Lemma}
 \newtheorem{prop}[thm]{Proposition} \theoremstyle{definition}  \newtheorem{defn}[thm]{Definition}
 \theoremstyle{remark} \newtheorem{rem}[thm]{Remark}   \numberwithin{equation}{section}
 \usepackage[top=2.2cm, bottom=2.2cm, left=2.2cm, right=2.2cm]{geometry}
\usepackage[framemethod=tikz]{mdframed}
\usepackage{lipsum}

\definecolor{mycolor}{rgb}{0.122, 0.835, 0.998}
\newmdenv[innerlinewidth=0.5pt, roundcorner=4pt,linecolor=mycolor,innerleftmargin=6pt,
innerrightmargin=6pt,innertopmargin=6pt,innerbottommargin=6pt]{mybox}
\newcommand{\norm}[1]{\left\Vert#1\right\Vert} \newcommand{\scal}[1]{\left<#1\right>}
\newcommand{\N}{\mathbb{N}}  \newcommand{\R}{\mathbb{R}}  \newcommand{\C}{\mathbb{C}}
   \newcommand{\BC}{\mathbb{T}} 
   
  \newcommand{\V}{\mathcal{V}} 
  \newcommand{\BargBC}{\mathcal{B}^{\sigma,\nu}_{\BC}}
            
\newcommand{\bz}{\overline{z}}   
     
 \newcommand{\eu}{{e_+}}  \newcommand{\es}{{e_-}}

\newcommand{\Mbcnumuo}{\mathcal{M}^{2,\sigma}(\BC^e_{\nu,\mu})}
\newcommand{\Mbcnumun}{\mathcal{M}^{2,\sigma}_n(\BC^e_{\nu,\mu})}
\newcommand{\MbcnumuT}{\mathcal{G}^{2,\sigma}(\BC^e_{\nu,\mu})}
\newcommand{\Mbcnumumn}{\mathcal{M}^{2,\sigma}_{m,n}(\BC^e_{\nu,\mu})}

\begin{document}

\title[On bicomplex  Fourier--Wigner transforms]{On bicomplex  Fourier--Wigner transforms}
\author{A. El Gourari}
\address{ Department of Mathematics, Faculty of Sciences,
	\newline Ibn Tofa\"il University, Kenitra}
\email{aiadelgourari@gmail.com}

\author{A. Ghanmi}
\author{K. Zine}
\address{A.G.S.-L.A.M.A., CeReMAR, Department of Mathematics,
           \newline P.O. Box 1014,  Faculty of Sciences,
           \newline Mohammed V University in Rabat, Morocco}

\email{allal.ghanmi@um5.ac.ma} 
\email{zine0khalil@gmail.com}


\begin{abstract}
	We consider the $1$- and $2$-d bicomplex analogs of the classical  Fourier--Wigner transform. Their basic properties, including Moyal's identity and characterization of their ranges giving rise to new  bicomplex--polyanalytic functional spaces are discussed. Particular case of special window is also considered. An orthogonal basis for the space of bicomplex--valued square integrable functions on the bicomplex numbers is constructed by means of the polyanalytic complex Hermite functions.
	\end{abstract}

\maketitle

\section{\texttt{Introduction}} \label{s1}

The standard Fourier--Wigner (windowed Fourier) transform is the well--defined bilinear mapping $\V: (f,g) \longmapsto  \V(f,g)$ on $L^2_\C(\R^d)\times L^2_\C(\R^d)$ given by the cross--Wigner function \cite{Folland1989,Thangavelu1993,Wong1998,Cohen2013,deGosson2017}
\begin{equation}
\V(f,g)(p,q) 
=\left(\frac{1}{2\pi}\right)^{\frac{d}{2}}
\int_{\R^d}  e^{i \scal{x-\frac{p}{2},q}_{\R^d}} f(x) \overline{g(x-p)} dx \label{FWT}
\end{equation}
for every $(p,q)\in \R^d \times \R^d$.
For fixed window state, it is closely related to Gabor's
transform \cite{Gabor1946}  
as well as to the well--known Segal--Bargmann transform \cite{Folland1989,Thangavelu1993}. It reduces to the familiar Wigner distribution when $f=g$; see e.g \cite{Wong1998,deGosson2017}.
The transform $\V$ has being intensively considered in harmonic analysis \cite{Folland1989,Thangavelu1993} and time--frequency analysis \cite{Cohen1995,Grocheng2001}.
In fact, it is very useful in in the study of nonexisting joint probability distribution of positioned momentum in a given state \cite{Wong1998}.
 It is a basic tool for 
interpreting quantum mechanics as a form of nondeterministic statical dynamics \cite{Moyal1949}.

The aim of this paper is to extend this transform to the bicomplex setting, i.e. where $(\R\times \R) \eu + (\R\times \R) \es$ is considered instead of the standard phase (time-frequency) space  $\R\times \R$.
 Although this can be accomplished in a number of different ways,
 we shall confine our attention to two main  natural bicomplex Fourier--Wigner transforms  (Sections 3 and 4).
We investigate their basic properties such as the corresponding Moyal's identities (energy preservation principle).
We also characterize their ranges leading to new bicomplex--polyanalytic functional spaces. We also provide a new class of four--indices bivariate complex orthogonal polynomials of Hermite type that form an orthonormal basis of the infinite Hilbert space of bicomplex--valued square integrable functions on bicomplex space (see Section 4).

We will rely mostly on the notations and basic tools as given in \cite{Zine2018} and
 relevant to bicomplex numbers $\BC$, bicomplex holomorphic functions and bicomplex Hilbert spaces,  For further detail, we can refer to  \cite{Price1991,RochonTremmblay2006,Zine2018} and the references therein.

 \section{Preliminaries: The rescaled Fourier--Wigner transform.}

  We begin by reviewing the notion and the basic facts related to the rescaled  Fourier--Wigner transform
 \begin{equation}\label{FWTsig}
 \V^\sigma(f,g)(p,q)=\left(\frac{\sigma}{2\pi}\right)^{\frac{d}{2}}
 \int_{\R^d}  e^{i \sigma \scal{x-\frac{p}{2},q}_{\R^d}} f(x) \overline{g(x-p)} dx.
 \end{equation}
 Such transform can be rewritten in terms of the translation operator $T_x$ and the modulation operator $M^\sigma_\xi$ given respectively by
  $T_x g(t) := g(t - x)$ and $M^\sigma_\xi g(t) = e^{i\sigma \xi\cdot t}g(t)$.
  In fact, we have
\begin{align*}
 \V^\sigma(f,g)(p,q)
 &= \left(\frac{\sigma}{2\pi}\right)^{\frac{d}{2}}  e ^ { -i\frac{\sigma}{2}
	\scal{p,q}_{\R^d} } \scal{f,M^\sigma_{-q} T_pg}_{L^2_{\C}(\R^d)}
\end{align*}
where $\scal{p,q}_{\R^d}$ denotes the usual scalar product in $\R^d$ and $L^2_{\C}(\R^d)$ is the space of $\C$--valued square integrable functions with respect to the Lebesgue measure $dx$ on $\R^d$.
Such transform maps $L^2_{\C}(\R^d) \times L^2_{\C}(\R^d)$ into $L^2_{\C}(\C^d)$ (see e.g. \cite{Thangavelu1993,Wong1998}).
An interesting result satisfied by $\V^\sigma$ is the Moyal's formula
\begin{eqnarray}
 \scal{\V^\sigma(f,g),\V^\sigma(\varphi,\psi)}_{L^2_{\C}(\C^d)}
&=& \scal{f,\varphi}_{L^2_{\C}(\R^d)} \scal{\psi,g}_{L^2_{\C}(\R^d)} \label{Moyal}
\end{eqnarray}
 for all $f,g,\varphi,\psi \in L^2_{\C}(\R^d)$.
It readily follows from the classical  Moyal's formula for $\V$ (\cite{Folland1989,Thangavelu1993,Wong1998,deGosson2008}) combined with the fact that $\V^\sigma(f,g)(p,q) = \sigma^{d/2} \V(f,g)(p,\sigma q)$. 
It interprets the fact that $\V^\sigma$ preserves the energy of a signal.
Accordingly, it can be shown \cite{Wong1998,deGosson2008} that the Fourier--Wigner transform $\V^\sigma$
produces orthonormal bases for the Hilbert space $L^2_{\C}(\C^d)$ from the ones of $L^2_{\C}(\R^d)$.
More precisely, if $\{\varphi_{k},k\in\N\}$ is an orthonormal basis of $L^2_{\C}(\R^d)$, then  $\{\varphi_{jk}=\V^\sigma(\varphi_{j},\varphi_{k}); \, j,k=0,1,2,\cdots \}$ is an orthonormal basis of $L^2_{\C}(\C^d)$.
This fact will be used, when dealing with the special bicomplex Fourier--Wigner transform discussed in Section 4, in order to obtain bicomplex four--indices orthogonal polynomials of Hermite type that are not tensor product of the Hermite polynomials on $\R$.

The next result is the analog of Theorem 3.1 in \cite{ABEG2015} for the action of $\V^{\sigma}$ on the rescaled Hermite functions
$$h^\sigma_n(t) = -1)^{n} e^{\frac{\sigma}{2} t^2} \dfrac{d^n}{dt^n} \left(e^{-\sigma t^2}\right)  = \sqrt{\sigma}^n h_{n}(\sqrt{\sigma}t)$$
that form an orthogonal basis of $L^2_{\C}(\R)$. It asserts that $\V^{\sigma}(h^\sigma_m,h^\sigma_n)$ is closely connected to the univariate polyanalytic Hermite function
 \begin{align}\label{chp}
h^\alpha_{m,n}(z,\bz) :=
(-1)^{m+n}e^{\frac{\alpha}{2} |z|^2}\dfrac{\partial ^{m+n}}{\partial \bz^{m} \partial z^{n}} \left(e^{-\alpha |z|^2}\right) , \quad \alpha>0.
\end{align}
We denote  $h_{m,n}=h^1_{m,n}$.

\begin{prop} \label{actionVsig}
	We have
	\begin{eqnarray}
	\V^{\sigma}(h^\sigma_m,h^\sigma_n) (p,q) = (-1)^{n} \frac{2^{m+n}}{\sqrt{2}}   h^{\sigma/2}_{m,n} (z, \bz).
	\end{eqnarray}
\end{prop}

\begin{proof}
 A 
 straightforward computation starting from the definition of $\V^{\sigma}$ and $h^\sigma_n$ shows
	\begin{align*}
	\V^{\sigma}(h^\sigma_m,h^\sigma_n) (p,q)
	&	= \sqrt{\sigma}^{m+n}	\V(h_m,h_n) (\sqrt{\sigma}p,\sqrt{\sigma}q).
	\end{align*}
	Subsequently, by means of Theorem 3.1 in \cite{ABEG2015} combined with the fact that $$ h^\alpha_{m,n}(z,\bz) := \sqrt{\alpha}^{m+n}
	h_{m,n}(\sqrt{\alpha}z,\sqrt{\alpha}\bz),$$ we obtain
	\begin{align*}
	\V^{\sigma}(h^\sigma_m,h^\sigma_n) (p,q)
	&=\sqrt{\sigma}^{m+n}
(-1)^{n}\sqrt{2}^{m+n-1}
h_{m,n}\left(\frac{\sqrt{\sigma}}{\sqrt{2}}z,
\frac{\sqrt{\sigma}}{\sqrt{2}}\bz\right)\\
	&= (-1)^{n}\frac{2^{m+n}}{\sqrt{2}}   h^{\sigma/2}_{m,n} (z, \bz ).
	\end{align*}
\end{proof}

\begin{rem} \label{}
The range of $L^2_{\C}(\R)$ by the transform  $\V^{\sigma}(\cdot,h^\sigma_n)$  is the Hilbert space spanned by the complex Hermite functions $h^{\sigma/2}_{m,n} (z, \bz )$ for varying $m$, which is clearly connected to the so--called true--poly--Fock space $\mathcal{F}^{2,\sigma/2}_n(\C)$ on $\C$ of level $n$ in Vasilevski's terminology  \cite{Vasilevski2000,Abreu2010}.
\end{rem}

In the sequel, we suggest some natural bicomplex analogs of the Fourier--Wigner transform with input functions belonging to the bicomplex Hilbert spaces $L^2_{\BC}(X)$ with $X=\R$ or $\R^2$ and output functions in $L^2_{\BC}(\BC)$. We then investigate some of their basic properties, such as the Moyal's identity, the identification of their ranges, the connection to some bicomplex transforms of Segal--Bargmann type, among others.
The central idea in obtaining such bicomplex analogs is basically the idempotent decomposition of any $\varphi\in L^2_{\BC}(X)$ as  $ \varphi = \varphi^- \eu  + \varphi^- \es$
 with $\varphi^+,\varphi^- \in L^2_{\C}(X)$.

Throughout the rest of the paper, the notation $\C_\tau$ (with $\tau^2=-1$) will be used to mean the complex plane $\C_\tau:=\{z_\nu:=x+\tau y; \, x,y\in\R\}$ with $\C_i=\C$.

\section{Unidimensional bicomplex Fourier--Wigner transform.}

For every given bicomplex number $Z=z_1+jz_2$, where $z_\ell = x_\ell + iy_\ell$, we associate the hyperbolic numbers $X_e=x_1 \eu + x_2\es$ and $Y_e=y_1 \eu + y_2\es$,
 and consider the translation operator
$$T_{X_e} \varphi(t):
= \varphi^+ (t-x_1)\eu + \varphi^-(t-x_2) \es$$
as well as the modified modulation operator
$$M^{\sigma,\nu,\mu}_{X_e,Y_e} \varphi (t) := e^{\sigma(\nu \eu+\mu\es)\left(t - \frac{X_e}2 \right)Y_e} \varphi(t)$$
for given $\varphi = \varphi^+ \eu + \varphi^- \es \in L^2_{\BC}(\R)$.

\begin{defn}
	We call unidimensional bicomplex Fourier--Wigner transform the integral transform $\V^{\sigma,\nu,\mu}_{\R,\BC}$ on $L^2_{\BC}(\R)\times L^2_{\BC}(\R)$ defined by
\begin{eqnarray}\label{V1def}
\V^{\sigma,\nu,\mu}_{\R,\BC}(\varphi,\psi) (Z)
&:=& \left(\frac{\sigma}{\pi}\right)
e^{-\frac{\sigma}4 \left( (X_e^\dagger)^2 + (Y_e^\dagger)^2\right) } \int_{\R} \varphi (t)
 M^{\sigma,\nu,\mu}_{X_e,Y_e} \left(  T_{X_e} \psi(t) \right)^* dt
\end{eqnarray}
with $X_e^\dagger =x_2\eu + x_1 \es$ and $Z*=\bz_1-j\bz_2$.
\end{defn}

The following lemmas will play a crucial rule in establishing the main results of this section. To this end, we introduce $z_{1_\nu}= x_1+\nu y_1$ and $z_{2_\mu}= x_2+\mu y_2$ for given $Z=z_1+jz_2$ with $z_\ell = x_\ell + iy_\ell$; $\ell=1,2$.

\begin{lem}\label{lemMoydec} Let $\varphi,\psi \in L^2_{\BC}(\R)$. Then, we have the splitting formula 
\begin{eqnarray}  \label{V1decomp}
\V^{\sigma,\nu,\mu}_{\R,\BC}(\varphi,\psi) (Z)
&=&  \left(\frac{2\sigma}{\pi}\right)^{\frac{1}{2}}  e^{-\frac{\sigma}4 |z_{2_\mu}|^2} \V^{\sigma,\nu} (\varphi^+,\psi^+) (z_{1_\nu}) \eu
\\ & & \qquad + \left(\frac{2\sigma}{\pi}\right)^{\frac{1}{2}} e^{-\frac{\sigma}4 |z_{1_\nu}|^2}\V^{\sigma,\mu} (\varphi^-,\psi^-) (z_{2_\mu}) \es .\nonumber
\end{eqnarray}
Moreover, $\V^{\sigma,\nu,\mu}_{\R,\BC}(\varphi,\psi) $ belongs to $L^2_{\BC}(\BC)$ and we have the Moyal's identity
\begin{eqnarray}\label{MoyalL3}
\scal{\V^{\sigma,\nu,\mu}_{\R,\BC}(\varphi_1,\psi_1),
	\V^{\sigma,\nu,\mu}_{\R,\BC}(\varphi_2,\psi_2) } _{L^2_{\BC}(\BC)} =
\scal{\varphi_1,\varphi_2} _{L^2_{\BC}(\R)}
\scal{\psi_1,\psi_2} _{L^2_{\BC}(\R)}
\end{eqnarray}
for every $\varphi_\ell,\psi_\ell \in L^2_{\BC}(\R)$; $\ell=1,2$.
\end{lem}

 \begin{proof}
 The first assertion follows easily from \eqref{V1def} since
 $$(X_e^\dagger)^2 + (Y_e^\dagger)^2 = |z_{2_\mu}|^2\eu + |z_{1_\nu}|^2 \es,$$
 $$M^{\sigma,\nu,\mu}_{X_e,Y_e} \varphi (t) = e^{\nu\sigma\left(t - \frac{x_1}2 \right)y_1} \varphi^+(t)\eu +
 e^{\mu\sigma\left(t - \frac{x_2}2 \right)y_2} \varphi^-(t)\es  $$
  and
  $$T_{X_e} \psi(t) :=
  T_{x_1}\psi^+(t) \eu + T_{x_2} \psi^-(t)\es .$$
In order to prove the second assertion, we notice first that the function
$\V^{\sigma,\nu} (\varphi^+,\psi^+) (z_{1_\nu})$ belongs to the Hilbert space $L^2_{\C}(\C_\nu)$.
Accordingly, the function $e^{-\frac{\sigma}4 |z_{2_\mu}|^2} \V^{\sigma,\nu} (\varphi^+,\psi^+) (z_{1_\nu})$ belongs to  $L^2_{\C}(\C_\nu\times\C_\mu)$. 
 The same observation holds true for 
 $e^{-\frac{\sigma}4 |z_{1_\nu}|^2}\V^{\sigma,\mu} (\varphi^-,\psi^-) (z_{2_\mu})$.
This shows that $\V^{\sigma,\nu,\mu}_{\R,\BC}(\varphi,\psi)$ belongs to the Hilbert space
$$L^2_{\C}(\C_\nu)\eu + L^2_{\C}(\C_\mu)\es = L^2_{\BC}(\C_\nu\times\C_\mu)=4 L^2_{\BC}(\BC).$$

Now, by denoting the left--hand side of \eqref{MoyalL3} by $M(\varphi_{1,2},\psi_{1,2})$ and making use of \eqref{V1decomp}, we obtain
\begin{eqnarray}
M(\varphi_{1,2},\psi_{1,2})
& = & \left(\frac{\sigma}{2\pi}\right)
\left(\int_{\C}	e^{-\frac{\sigma}2  |\xi|^2} d\lambda(\xi)\right)
\left(
\scal{ \V^{\sigma,\nu} (\varphi^+_1,\psi^+_1)  ,  \V^{\sigma,\nu} (\varphi^+_2,\psi^+_2)
} _{L^2_{\C}(\C_\nu)} \eu
\right. \nonumber
\\ & & \qquad \qquad \qquad \qquad \qquad+
\left.
\scal{ \V^{\sigma,\mu} (\varphi^-_1,\psi^-_1) ,\V^{\sigma,\mu} (\varphi^-_2,\psi^-_2)
} _{L^2_{\C}(\C_\mu)} \es \right) .\nonumber
\end{eqnarray}
 Consequently, from \eqref{Moyal} we get
 \begin{eqnarray}
 M(\varphi_{1,2},\psi_{1,2})
 &=&
\scal{\varphi^+_1,\varphi^+_2} _{L^2_{\C}(\R)} \scal{\psi^+_1,\psi^+_2} _{L^2_{\C}(\R)} \eu
+ \scal{\varphi^-_1,\varphi^-_2} _{L^2_{\C}(\R)}
\scal{\psi^-_1,\psi^-_2} _{L^2_{\C}(\R)} \es \nonumber
\\ &=&
\left( \scal{\varphi^+_1,\varphi^+_2} _{L^2_{\C}(\R)}\eu
+ \scal{\varphi^-_1,\varphi^-_2} _{L^2_{\C}(\R)}\es\right)
 \nonumber
\\ & & \qquad \qquad \qquad\times
\left(  \scal{\psi^+_1,\psi^+_2} _{L^2_{\C}(\R)} \eu +
\scal{\psi^-_1,\psi^-_2} _{L^2_{\C}(\R)} \es\right)  \nonumber
\\&=&
\scal{\varphi_1,\varphi_2} _{L^2_{\BC}(\R)}
\scal{\psi_1,\psi_2} _{L^2_{\BC}(\R)}. \nonumber
\end{eqnarray}
This completes our check of \eqref{MoyalL3} and hence the one of Lemma \ref{lemMoydec}.
  \end{proof}

Another needed fact is the action of $\V^{\sigma,\nu,\mu}_{\R,\BC}$ on the elementary functions
$$f^\sigma_{m,n}(t) := h^\sigma_m(t) \eu + h^\sigma_n(t) \es.$$ Namely, we assert

\begin{lem} We have
	\begin{eqnarray}\label{actionmrsn}
	\V^{\sigma,\nu,\mu}_{\R,\BC}(f^\sigma_{m,n},f^\sigma_{r,s}) (Z)
	&=&  \left(\frac{\sigma}{\pi}\right)^{\frac{1}{2}}	(-1)^{r}
	2^{m+r}  e^{-\frac{\sigma}4 |z_{2_\mu}|^2} h^{\sigma/2}_{m,r} (z_{1_\nu}, \overline{z_{1_\nu}}) \eu
	 \\&& \qquad\qquad\qquad\qquad  + \left(\frac{\sigma}{\pi}\right)^{\frac{1}{2}} (-1)^{s} 2^{n+s} e^{-\frac{\sigma}4 |z_{1_\nu}|^2 }h^{\sigma/2}_{n,s} (z_{2_\mu}, \overline{z_{2_\mu}})\es. \nonumber
	\end{eqnarray}
\end{lem}

    \begin{proof}
 The result follows making use of  \eqref{V1decomp} and Proposition \ref{actionVsig}. Indeed, we have
\begin{align*}
	\V^{\sigma,\nu,\mu}_{\R,\BC}&(f^\sigma_{m,n},f^\sigma_{r,s}) (Z)
=  \V^{\sigma,\nu,\mu}_{\R,\BC}(h^\sigma_m \eu + h^\sigma_n \es, h^\sigma_r \eu + h^\sigma_s \es) (Z) \nonumber\\
	&=  \left(\frac{2\sigma}{\pi}\right)^{\frac{1}{2}}  e^{-\frac{\sigma}4 \left( |z_{2_\mu}|^2\eu +|z_{1_\nu}|^2\es\right) }  \left(  \V^{\sigma,\nu} (h^\sigma_m,h^\sigma_r) (z_{1_\nu}) \eu
	+\V^{\sigma,\mu} (h^\sigma_n,h^\sigma_s) (z_{2_\mu}) \es \right) \nonumber\\
	&=    \left(\frac{\sigma}{\pi}\right)^{\frac{1}{2}}  e^{-\frac{\sigma}4 \left( |z_{2_\mu}|^2\eu +|z_{1_\nu}|^2\es\right) }  \left( 	(-1)^{r} 2^{m+r} h^{\sigma/2}_{m,r} (z_{1_\nu}, \overline{z_{1_\nu}}) \eu 	\right. \\
	& \qquad\qquad\qquad \qquad \qquad \qquad \qquad \left.
	+ (-1)^{s} 2^{n+s} h^{\sigma/2}_{n,s} (z_{2_\mu}, \overline{z_{2_\mu}})\es\right) .
\end{align*}
   \end{proof}

Below, we will discuss the basic properties of the transform $ \V^{\sigma,\nu,\mu}_{\R,\BC}$ for the special window function $\psi_{0}(t):= e^{-\frac{\sigma}{2}t^{2}}$. To this end, we need to associate to a bicomplex number $Z= z_1+jz_2 \in\BC$, its companion $Z^e_{\nu,\mu}= z_{1_\nu}\eu + z_{2_\mu}\es$
and perform
$$\BC^e_{\nu,\mu} = \C_\nu \eu - \C_\mu\es=\{Z^e_{\nu,\mu} =(x_1+\nu y_1)\eu (x_2+\mu y_2)\es
, \,x_1, y_1,x_2,y_2\in\R \}.$$
Therefore, any bicomplex--valued function $f(Z)$ on $\BC$
can be seen as a function on $\BC^e_{\nu,\mu}$.

\begin{defn}
A bicomplex--valued function $f$ on $\BC$ is said to be $\BC^e_{\nu,\mu}$--holomorphic if its companion $ f^e(Z^e_{\nu,\mu}):= f(Z)$ is $\BC^e_{\nu,\mu}$--holomorphic in the sense that $f^e$ satisfies the system of first order differential equations
\begin{eqnarray}
\frac{\partial f^e}{\partial {Z^e_{\nu,\mu}}^*} = \frac{\partial f^e}{\partial \overline{Z^e_{\nu,\mu}}}  = \frac{\partial f^e}{\partial {Z^e_{\nu,\mu}}^\dagger}  = 0,
\end{eqnarray}
where
\begin{eqnarray*}
\frac{\partial}{\partial {Z^e_{\nu,\mu}}^*}= \partial_{\bz_{2_\mu}} \eu +\partial_{\bz_{1_\nu}} \es;
\frac{\partial}{\partial \overline{Z^e_{\nu,\mu}}}
= \partial_{\bz_{1_\nu}} \eu + \partial_{\bz_{2_\mu}} \es;
\frac{\partial}{\partial {Z^e_{\nu,\mu}}^\dagger}=  \partial_{z_{2_\mu}} \eu + \partial_{z_{1_\nu}} \es.
\end{eqnarray*}
\end{defn}

This is clearly equivalent to rewrite $f$ in the form
$$f(Z)=f^e(Z^e_{\nu,\mu}) = F(z_{1_\nu}) \eu + G(z_{2_\mu}) \es $$
with $F\in \mathcal{H}ol(\C_\nu)$ (resp $G\in \mathcal{H}ol(\C_\mu)$) is a holomorphic function on $ \C_\nu$ (resp. $\C_\mu$).
A variant bicomplex Bargmann space of the one introduced in \cite{Zine2018} is the following.

\begin{defn}
	We call compagnion $\BC^e_{\nu,\mu}$--Bargmann space, the Hilbert space
		 $\mathcal{F}^{2,\sigma}(\BC^e_{\nu,\mu})$ of all bicomplex--valued $\BC^e_{\nu,\mu}$--holomorphic functions $ f(Z) =f^e(Z^e_{\nu,\mu})= F(z_{1_\nu}) \eu + G(z_{2_\mu}) \es$ such that
	$F\in L^{2,\sigma/2}_{\C}(\C_\nu)$ and $G\in L^{2,\sigma/2}_{\C}(\C_\mu)$.
	Succinctly,	
	$$\mathcal{F}^{2,\sigma}(\BC^e_{\nu,\mu}) = \mathcal{F}^{2,\sigma/2}(\C_\nu) \eu + \mathcal{F}^{2,\sigma/2}(\C_\mu) \es,$$ where $\mathcal{F}^{2,\sigma/2}(\C_\tau)$ denotes the classical complex Bargmann space of weight $\sigma/2$ on $\C_\tau$.
	\end{defn}

This functional space is trivially endowed with the bicomplex scaler product
\begin{eqnarray}\label{BcSP}
\scal{f_1,f_2}_{\BC^e_{\nu,\mu}}
&=& \scal{F_1,F_2}_{L^{2,\sigma/2}_{\C}(\C_\mu)} \eu +   \scal{F_1,F_2}_{L^{2,\sigma/2}_{\C}(\C_\mu)} \es
\end{eqnarray}
for given $
f_\ell(Z) = F(z_{1_\nu}) \eu +   G(z_{2_\mu}) \es  $.
Accordingly, it can be seen as subspace of
$L^2_{\BC}(\C_\nu\times\C_\mu)= 4L^2_{\BC}(\BC)$ by considering
its range $\Mbcnumuo:=  M_{\sigma/2}(\mathcal{F}^{2,\sigma}(\BC^e_{\nu,\mu}))$ by the multiplication operator
 	\begin{eqnarray}
 	M_{\sigma/2} f (Z) &: =&  e^{-\frac{\sigma}2 |Z^e_{\nu,\mu}|^2} f^e (Z^e_{\nu,\mu}) \nonumber \\&=& e^{-\frac{\sigma}4\left( |z_{1_\nu}|^2+ |z_{2_\mu}|^2\right) } \left(  F(z_{1_\nu}) \eu + G(z_{2_\mu}) \es\right)  .\nonumber
  	\end{eqnarray}
 In fact, for given $\Upsilon_\ell= M_{\sigma/2} f_\ell \in \Mbcnumuo$; $\ell=1,2$,
 we have
 	\begin{eqnarray}
 \scal{\Upsilon_1,\Upsilon_2}_{L^2_{\BC}(\BC)}=
 \int_{\BC} \Upsilon_1(Z) (\Upsilon_2(Z))^* d\lambda(Z) = \frac 14 \left( \frac{2\pi}{\sigma}\right) \scal{f_1,f_2}_{\BC^e_{\nu,\mu}}. \nonumber
 \end{eqnarray}
The corresponding bicomplex norm is the one given through
	\begin{eqnarray}\label{bicnorm}
\norm{\Upsilon}^2_{\BC^e_{\nu,\mu}}: = \left( \frac{\pi}{4\sigma}\right) \left( \norm{F}^2_{L^{2,\sigma/2}_{\C}(\C_\nu)} + \norm{G}^2_{L^{2,\sigma/2}_{\C}(\C_\mu)} \right).
\end{eqnarray}
Thus, we claim the following

\begin{prop}
	The space $\Mbcnumuo$ is a reproducing
kernel bicomplex Hilbert space whose kernel fuction is given by $$
K_\sigma (Z^e_{\nu,\mu},W^e_{\nu,\mu}) =  e^{-\frac{\sigma}2 \left( |Z^e_{\nu,\mu}|^2 + |W^e_{\nu,\mu}|^2 + Z^e_{\nu,\mu}(W^e_{\nu,\mu})^*\right) }   .$$
\end{prop}

 Moreover, we prove

\begin{thm} \label{Vsigfmn}
	The transform $\mathcal{S}^{\sigma,\nu,\mu}_0 (\varphi)$ given by
	$$\mathcal{S}^{\sigma,\nu,\mu}_0 (\varphi) := \left(\frac{\sigma}{\pi}\right)^{\frac{1}{4}} \V^{\sigma,\nu,\mu}_{\R,\BC}(\varphi, \psi_{0})$$
	 defines an isometry from $ L^2_{\BC}(\R)$ onto the Hilbert space $ \Mbcnumuo$.
	Moreover, the functions
	\begin{eqnarray}\label{Fctrangenn}
	\varphi^{\sigma,\nu,\mu}_n(Z)
	&=& \left(\frac{\sigma}{\pi}\right)^{\frac{3}{4}} \sigma^{n}
	\left(Z^e_{\nu,\mu}\right)^n  e^{ -\frac{\sigma}{2} |Z^e_{\nu,\mu}|^2  }
	\end{eqnarray}
	form an orthogonal basis of $ \Mbcnumuo$ with norm given by
\begin{eqnarray}\label{normFctrangenn}
\norm{ \varphi^{\sigma,\nu,\mu}_n}^2_{L^2_{\BC}(\BC)}=
\left(\frac\pi \sigma\right)^{\frac{1}{2}}  2^n \sigma^n n! .
\end{eqnarray}
\end{thm}

\begin{proof}
	 Notice first that the window state is the Gaussian centred at the origin $\psi_{0}(t):= e^{-\frac{\sigma}{2}t^{2}}=f_{0,0}(t)$ and $h^\sigma_n=f^\sigma_{n,n}$ for $\eu+\es=1$. Thus, from \eqref{actionmrsn} and the fact $h^\alpha_{n,0}(\xi,\overline{\xi})= \alpha^n \xi^b e^{-\frac{\alpha}{2}|\xi|^2}$, we obtain
	\begin{eqnarray}
	\V^{\sigma,\nu,\mu}_{\R,\BC}(h^\sigma_n, \psi_{0}) (Z)
	&=&\V^{\sigma,\nu,\mu}_{\R,\BC}(f^\sigma_{n,n}, f_{0,0}) (Z) \nonumber
	\\&=&
		\left(\frac{\sigma}{\pi}\right)^{\frac{1}{2}} 2^{n}
	 \left(  e^{-\frac{\sigma}4 |z_{2_\mu}|^2} h^{\sigma/2}_{n,0} (z_{1_\nu}, \overline{z_{1_\nu}}) \eu +   e^{-\frac{\sigma}4 |z_{1_\nu}|^2} h^{\sigma/2}_{n,0} (z_{2_\mu}, \overline{z_{2_\mu}})\es \right) .  \nonumber
\\&=&
	\left(\frac{\sigma}{\pi}\right)^{\frac{1}{2}} \sigma^{n}
\left( z_{1_\nu}^n \eu +   z_{2_\mu}^n \es \right)  e^{ -\frac{\sigma}{4}\left( \left( |z_{2_\mu}|^2 + |z_{1_\nu}|^2\right) \eu    +\left( |z_{1_\nu}|^2 + |z_{2_\mu}|^2\right) \right) \es  } \nonumber
\\&=&
\left(\frac{\sigma}{\pi}\right)^{\frac{1}{2}} \sigma^{n}
\left( z_{1_\nu}^n \eu +   z_{2_\mu}^n \es \right)  e^{ -\frac{\sigma}{2} |Z^e_{\nu,\mu}|^2  }.
\label{rangeV}
	\end{eqnarray}
These functions form clearly an orthogonal system in the Hilbert space $L^2_{\BC}(\BC)$ in virtue of the Moyal's identity \eqref{MoyalL3} satisfied by $\V^{\sigma,\nu,\mu}_{\R,\BC}$ and the orthogonality of $h^\sigma_n$ in $ L^2_{\BC}(\R)$. Indeed, we have
\begin{eqnarray}
\scal{\varphi^{\sigma,\nu,\mu}_m, \varphi^{\sigma,\nu,\mu}_n}_{L^2_{\BC}(\BC)}
&=&
\left(\frac{\sigma}{\pi}\right)^{\frac{1}{2}} \scal{\V^{\sigma,\nu,\mu}_{\R,\BC}(h^\sigma_m, \psi_{0}), \V^{\sigma,\nu,\mu}_{\R,\BC}(h^\sigma_n, \psi_{0})}_{L^2_{\BC}(\BC)} \nonumber
\\&=&
\left(\frac{\sigma}{\pi}\right)^{\frac{1}{2}}\scal{h^\sigma_m, h^\sigma_n}_{L^2_{\BC}(\R)} \scal{ \psi_{0}, \psi_{0}}_{L^2_{\BC}(\R)}\nonumber
\\&=&
\left(\frac{\sigma}{\pi}\right)^{\frac{1}{2}}\norm{\psi_{0}}^2_{L^2_{\C}(\R)}\norm{h^\sigma_n}^2_{L^2_{\C}(\R)} \delta_{m,n}\nonumber
\\&=&
\norm{h^\sigma_n}^2_{L^2_{\C}(\R)} \delta_{m,n}.\nonumber
\end{eqnarray}
This readily follows since $\norm{\psi_{0}}^2_{L^2_{\BC}(\R)}= \left(\frac\pi \sigma\right)^{1/2} $ and consequently gives rise to \eqref{normFctrangenn} for $\norm{h^\sigma_n}^2_{L^2_{\BC}(\R)} = \left(\frac\pi \sigma\right)^{1/2} 2^n \sigma^n n!$. Identity \eqref{normFctrangenn} can also be handled by direct computation using the explicit expression of $\varphi^{\sigma,\nu,\mu}_n$.
The previous result remains valid for any $\varphi \in L^2_{\BC}(\R)$. Indeed,  by applying the Moyal's identity \eqref{MoyalL3}, we get
\begin{align*}
\norm{\mathcal{S}^{\sigma,\nu,\mu}_0 (\varphi) }^2_{L^2_{\BC}(\BC)}
&= \left(\frac \sigma\pi\right)^{1/2} \left|\scal{\V^{\sigma,\nu,\mu}_{\R,\BC}(\varphi, \psi_{0}) ,\V^{\sigma,\nu,\mu}_{\R,\BC}(\varphi, \psi_{0})}_{L^2_{\BC}(\BC)}\right|
\\&= \left(\frac \sigma\pi\right)^{1/2}\left|\scal{\varphi,\varphi}_{L^2_{\BC}(\R)}
\scal{ \psi_{0}, \psi_{0}}_{L^2_{\BC}(\R)} \right|
\\ &=   \norm{\varphi}^2_{L^2_{\BC}(\R)}.
\end{align*}
This shows in particular that $\mathcal{S}^{\sigma,\nu,\mu}_0   \in \Mbcnumuo$.
One can conclude for the proof, by noting that the functions $ \varphi^{\sigma,\nu,\mu}_n(Z) $ in \eqref{Fctrangenn} is a complete orthogonal system in $\Mbcnumuo$ for the monomials $\left(Z^e_{\nu,\mu}\right)^n$ form an orthogonal basis of $L^{2,\sigma/2}_{\BC}(\C_\nu\times\C_\mu)$.
Moreover, for any
 $\displaystyle \varphi(t)= \sum_{n=0}^\infty c_n h^\sigma_n \in L^2_{\BC}(\R),$
we have
$$ \mathcal{S}^{\sigma,\nu,\mu}_0 (\varphi)
= \left(\frac \sigma\pi\right)^{1/4} \sum_{n=0}^\infty c_n  \varphi^{\sigma,\nu,\mu}_n(Z)
$$ which follows by means of \eqref{rangeV} and the continuity of the linear mapping $\mathcal{S}^{\sigma,\nu,\mu}_0$.
\end{proof}

\begin{cor} \label{V11}
	The transform $\mathcal{S}^{\sigma,\nu,\mu}_0$
  is closely connected to the bicomplex Segal--Bargmann transform $\BargBC$ introduced in \cite{Zine2018}. More precisely, we have
	\begin{eqnarray}\label{BicSBT}
\mathcal{S}^{\sigma,i,i}_0(\varphi)(Z)	= \left(\frac{\sigma}{\pi}\right)
	e^{-\frac{\sigma}2   |Z^e_{\nu,\mu}|^2  }
	\int_{\R}    e^{-\sigma \left( t - \frac{Z^e_{\nu,\mu}}{2}\right)^2 }   e^{\frac{\sigma}2 t^2 } \varphi(t)  dt  .
	\end{eqnarray}
\end{cor}

\begin{proof}
	Identity \eqref{BicSBT} which follows by a tedious but straightforward computation. Indeed, we obtain
	\begin{eqnarray} 
		\mathcal{S}^{\sigma,\nu,\mu}_0(\varphi)(Z)
	&=&\left(\frac{\sigma}{\pi}\right)
	e^{-\frac{\sigma}2   |Z^e_{\nu,\mu}|^2  }
	\int_{\R}    e^{-\sigma \left( t - \frac{Z^e_{\nu,\mu}}{2}\right)^2 }   e^{\frac{\sigma}2 t^2 } \varphi(t)  dt , \nonumber
\end{eqnarray}
so that for $\nu=\mu=i$, we recover the bicomplex Segal--Bargmann transform introduced in \cite[Eq. (5.6) ]{Zine2018} (with $\nu=\sigma$ there) from  $ L^2_{\BC}(\R)$ onto the bicomplex Bargmann space
$\mathcal{CF}^{2,\sigma}(\BC^e_{i,i})$. 
\end{proof}

The last result of this section identifies the range $ \V^{\sigma,\nu,\mu}_{\R,\BC}(L^2_{\BC}(\R)\times L^2_{\BC}(\R))$ as special bicomplex--analytic closed subspace of $L^2_{\BC}(\BC)$.

\begin{defn}
	We call bicomplex $(n^*,1^{-},1^\dagger)$--$\BC^e_{\nu,\mu}$-- companion Bargmann space of first kind the Hilbert space $\mathcal{F}^{2,\sigma,\nu,\mu}_n(\BC^e_{\nu,\mu})$ of all bicomplex--valued functions
	$f(Z) = f^e(Z^e_{\nu,\mu})= F(z_{1_\nu}) \eu + G(z_{2_\mu}) \es $ satisfying the system
\begin{eqnarray}
\frac{\partial^{n+1}  f^e}{\partial [{Z^e_{\nu,\mu}}^*]^{n+1}}
= \frac{\partial  f^e}{\partial \overline{(Z^e_{\nu,\mu})}} =
 \frac{\partial  f^e}{\partial (Z^e_{\nu,\mu})^\dagger} = 0,
\end{eqnarray}
	and $\norm{F}^{2}_{L^{2,\sigma/2}_{\C}(\C_\nu)}$ and $\norm{G}^2_{L{2,\sigma/2}_{\C}(\C_\mu)} $ are finite.
\end{defn}

Thus, we claim the following (we omit the proof for its similarity to one provided above in the case $n=0$).

\begin{lem}\label{lemDecBicn}
The spaces $\Mbcnumun:= M_{\sigma/2} (\mathcal{F}^{2,\sigma,\nu,\mu}_n(\BC^e_{\nu,\mu})) $
are closed subspaces of $L^2_{\BC}(\BC)$ and we have
\begin{align} \label{DecompBFS}
\Mbcnumun =
 	 e^{-\frac{\sigma}2  |Z^e_{\nu,\mu}|^2 }  \left( \mathcal{F}^{2,\sigma/2}_n(\C_\nu) \eu +     \mathcal{F}^{2,\sigma/2}_n(\C_\mu)\es \right) .
\end{align}
Moreover, they are pairewisely orthogonal in $L^2_{\BC}(\BC)$.
\end{lem}

The decomposition \eqref{DecompBFS} must be understood in the sense that the bicomplex--valued function $f\in \mathcal{F}^{2,\sigma,\nu,\mu}_n(\BC^e_{\nu,\mu})$ is of the form
$$
f(Z)= F(z_{1_\nu})\eu + G(z_{2_\mu})\es
$$
with $F$ and $G$ are the $\C$--valued functions belonging to $\mathcal{F}^{2,\sigma}_n(\C_\nu)$ and $\mathcal{F}^{2,\sigma}_n(\C_\mu)$, respectively.
We endow $\mathcal{F}^{2,\sigma,\nu,\mu}_n(\BC^e_{\nu,\mu})$ with the bicomplex scalar product $\scal{\cdot,\cdot}_{\BC^e_{\nu,\mu}}$ in \eqref{BcSP}. The associated bicomplex norm is given by \eqref{bicnorm}.

\begin{prop}\label{PropnBCBarg}
	The functions $\psi^{\sigma,\nu,\mu}_{m,n}(Z^e_{\nu,\mu},{Z^e_{\nu,\mu}}^*)
	= e^{-\frac{\sigma}2 |Z^e_{\nu,\mu}|^2 } H^{\sigma}_{m,n}(Z^e_{\nu,\mu},{Z^e_{\nu,\mu}}^*) $, where
	\begin{align}\label{basisnBCBargspace}
			h^{\sigma}_{m,n}(Z^e_{\nu,\mu},{Z^e_{\nu,\mu}}^*) :=
			(-1)^{m+n} 	e^{\frac{\sigma}4 Z^e_{\nu,\mu} {Z^e_{\nu,\mu}}^*  }
			\frac{\partial^{m+n}}{\partial ({Z^e_{\nu,\mu}}^*)^m \partial(Z^e_{\nu,\mu})^n} \left(
				e^{-\frac{\sigma}2 Z^e_{\nu,\mu}{Z^e_{\nu,\mu}}^* }\right)
	\end{align}
	for varying $m$, form an orthogonal basis of the infinite $\BC^e_{\nu,\mu}$-Hilbert space $\Mbcnumun$.
\end{prop}

\begin{proof}
The proof is similar to one provided for $n=0$, but here we make use of the fact that the complex Hermite functions $h^{\sigma/2}_{m,n}(\xi)$ is an orthogonal basis of  $L^{2,\sigma/2}_{\C}(\C_\tau)$ and that
$$ H^{\sigma/2}_{m,n}(z_{1_\nu}) \eu +  H^{\sigma/2}_{m,n}(z_{2_\mu}) \es
= 	H^{\sigma}_{m,n}(Z^e_{\nu,\mu},{Z^e_{\nu,\mu}}^*). $$
\end{proof}

\begin{thm}
	The transform
	$$ \mathcal{S}^{\sigma,\nu,\mu}_n f :=\frac{\left(\frac{\sigma}{\pi} \right)^{1/4}}{\sqrt{2^n\sigma^n n!}} \V^{\sigma,\nu,\mu}_{\R,\BC}(f,h^\sigma_n)$$ corresponding to the window function $h^\sigma_n$ defines an isometry from $L^2_{\BC}(\R)$ onto the Hilbert space $\Mbcnumun$.
\end{thm}

\begin{proof}
	The proof can be handled in a similar way as for Theorem \ref{Vsigfmn} (for $n=0$) with $h^\sigma_0=\psi_0$. Let just mention that the expression of the functions $\V^{\sigma,\nu,\mu}_{\R,\BC}(h^\sigma_m,h^\sigma_n)$ is a particular of \eqref{actionmrsn}. Thus,
	\begin{eqnarray}
	\V^{\sigma,\nu,\mu}_{\R,\BC}(h^\sigma_m,h^\sigma_n) (Z)
	&=&
	\V^{\sigma,\nu} (f^\sigma_{m,m},f^\sigma_{n,n}) (Z) 	\nonumber \\
	&=&(-1)^{n} 2^{m+n} \left(\frac{\sigma}{\pi}\right)^{\frac{1}{2}} 		
		e^{-\frac{\sigma}2 |Z^e_{\nu,\mu}|^2 }\left(  h^{\sigma/2}_{m,n}(z_{1_\nu}) \eu +  h^{\sigma/2}_{m,n}(z_{2_\mu}) \es\right)   , \nonumber \\
		&=&(-1)^{n} 2^{m+n} \left(\frac{\sigma}{\pi}\right)^{\frac{1}{2}}
		\psi^{\sigma,\nu,\mu}_{m,n}(Z^e_{\nu,\mu},{Z^e_{\nu,\mu}}^*) 	
	\end{eqnarray}
	where $\psi^{\sigma,\nu,\mu}_{m,n}$ are as in Proposition \ref{PropnBCBarg}. 
		The range of $L^2_{\BC}(\R)$ by $\mathcal{S}^{\sigma,\nu,\mu}_n$   is then spanned by the bicomplex Hermite functions $\psi^{\sigma,\nu,\mu}_{m,n}$ for varying $m$ ($n$ fixed).
	Thus, one can conclude making use of Proposition \ref{PropnBCBarg} and the Moyal's identity \eqref{MoyalL3}.
\end{proof}

\begin{thm}
	The transform  $ \V^{\sigma,\nu,\mu}_{\R,\BC}$ defines an isometry from $L^2_{\BC}(\R)\times L^2_{\BC}(\R)$ onto the Hilbert space $$\MbcnumuT:=\bigoplus_{n=0}^{+\infty} \Mbcnumun.$$
\end{thm}

\begin{proof}
By  Lemma \ref{lemDecBicn}, it is clear that the bicomplex Hermite functions $\psi^{\sigma,\nu,\mu}_{m,n}$ for varying $m$ and $n$ form an orthogonal basis of the range of $L^2_{\BC}(\R)\times L^2_{\BC}(\R)$ by $ \V^{\sigma,\nu,\mu}_{\R,\BC}$.
\end{proof}

\begin{rem}
The space $\Mbcnumuo$ is strictely contained in $ L^2_{\BC}(\BC)$ since the functions
$$ \varphi^{\sigma,\nu,\mu}_{m,n}(Z)
=
\left(\sigma^{m} z_{2_\mu}^m \eu +  \sigma^{n} z_{1_\nu}^n \es \right)  e^{ -\frac{\sigma}{2} |Z^e_{\nu,\mu}|^2  }
$$
 belong to $ L^2_{\BC}(\BC)$ whenever $m\ne n$ but do not belongs to $\MbcnumuT$.
\end{rem}

\section{Bidimensional bicomplex Fourier--Wigner transform.}

In this section, we consider the natural extention to the bicomplex Hilbert space  $L^2_{\BC}(\R^2)$ of the operators defined on $L^2_{\C}(\R^2)$ by
$$M^{\nu,\sigma}_{X,Y} g(U) = e^{\nu \sigma \scal{U-\frac{X}2,Y}}g(U) \quad \mbox{and} \quad T_X g(U) := g(U - X)$$
where $X,Y\in\R^2$. Namely, we define
$$\widetilde{M^{\sigma,\nu,\mu}_{X,Y}}  \varphi = M^{\sigma,\nu}_{X,Y} \varphi^+ \eu +  M^{\sigma,\mu}_{X,Y} \varphi^- \es$$
and 
$$\widetilde{T_X} \psi(U) := \psi^+(U - X)\eu +  \psi^-(U - X)\es
 $$
for given $ \varphi = \varphi^+  \eu +  \varphi^- \es $ and $ \psi = \psi^+  \eu +  \psi^- \es $ in $L^2_{\BC}(\R^2)$ with $\varphi^+, \varphi^-, \psi^+ , \psi^- \in L^2_{\C}(\R^2)$.

\begin{defn}
	We call bidimensional bicomplex Fourier--Wigner transform that we denote by $\V^{\sigma,\nu,\mu}_{\R^2,\BC}$ the one associated to
	the "bicomplex time--frequency shift" operator
	$\widetilde{M^{\sigma,\nu,\mu}_{X,-Y}}\widetilde{  T_X}$ on $L^2_{\BC}(\R^2)\times L^2_{\BC}(\R^2)$ 
	and given explicitely by
	\begin{equation}
	\V^{\sigma,\nu,\mu}_{\R^2,\BC}(\varphi,\psi) (Z) = \left(\frac{1}{2\pi}\right)^{\frac{1}{2}} \int_{\R^2} e^{\sigma(\nu\eu+\mu\es) \scal{U-\frac{X}2,Y}_{\R^2} } \varphi(U) \psi^{*}(U-X) d\lambda(U) \label{FWdecomp2}
	\end{equation}
	with $Z=z_1+jz_2\in \BC$, $z_\ell=x_\ell+iy_\ell$, $X=(x_1,x_2)$ and $Y=(y_1,y_2)$.
\end{defn} 

By proceeding in a similar way as in the previous section, we can prove the following (we omit the proof).

\begin{lem}
	We have
	 	\begin{equation}\label{BidV1}
	\V^{\sigma,\nu,\mu}_{\R^2,\BC}(\varphi,\psi) (Z) = \V^{\sigma,\nu}_{\R^2,\C^2}(\varphi^+,\psi^+) (X,Y) \eu + \V^{\sigma,\mu}_{\R^2,\C^2}(\varphi^-,\psi^-) (X,Y) \es
	\end{equation}
	as well as
		\begin{eqnarray}\label{Moyal4}
		\scal{\V^{\sigma,\nu,\mu}_{\R^2,\BC}(\varphi_1,\psi_1) , \V^{\sigma,\nu,\mu}_{\R^2,\BC}(\varphi_2,\psi_2)}_{L^2_{\BC}( \BC)}
		=
		\scal{\varphi_1, \varphi_2 }_{L^2_{\BC}( \C)}    \scal{\psi_1,\psi_2 }_{L^{2}_{\BC}( \C)} .
		\end{eqnarray}
\end{lem}

\begin{prop} \label{MainThm3a}
	The bicomplex Fourier--Wigner transform $\V^{\sigma,\nu,\mu}_{\R^2,\BC}$ defines a surjection from $L^2_{\BC}(\R^2)\times L^2_{\BC}(\R^2)$ onto $L^2_{\BC}(\BC)$.
\end{prop}

\begin{proof}
	Let $F\in L^2_{\BC}(\BC)$. Then, we can rewrite $F$ as
	 $F=F^+\eu + F^-\es$ for certain $F^\pm\in L^2_{\C}(\C^2)$.
	 	By the surjectivity of $\V^{\sigma,\nu}_{\R^2,\C^2}$ and $\V^{\sigma,\mu}_{\R^2,\C^2}$ from $L^2_{\C}(\R^2)\times L^2_{\C}(\R^2)$ onto $L^2_{\C}(\C^2)$, we can exhibit
	 $\varphi^\pm ,\psi^\pm \in L^2_{\C}(\R^2)$ such that
	$$F^+(Z) = \V^{\sigma,\nu}_{\R^2,\C^2}(\varphi^+ ,\psi^+) (X,Y)$$
	and
	$$ F^-(Z) = \V^{\sigma,\mu}_{\R^2,\C^2}(\varphi^- ,\psi^-) (X,Y).$$
	Accordingly,
	$$
	F(Z) 	= \V^{\sigma,\nu}_{\R^2,\C^2}(\varphi^+ ,\psi^+) (X,Y) \eu +
	\V^{\sigma,\nu}_{\R^2,\C^2}(\varphi^- ,\psi^-) (X,Y) \es.
	$$
In virtue of \eqref{BidV1} and setting  $\varphi := \varphi^+\eu+\varphi^-\es$ and $\psi:=\psi^+\eu+\psi^-\es$, we get
$$ F(Z)	 = \V^{\sigma,\nu,\mu}_{\R^2,\BC}(\varphi^+\eu+\varphi^-\es,\psi^+\eu+\psi^-\es) (Z) = \V^{\sigma,\nu,\mu}_{\R^2,\BC}(\varphi,\psi) (Z).$$
Notice finally that $\varphi,\psi\in L^2_{\BC}(\R^2)$ since $\varphi^\pm ,\psi^\pm \in L^2_{\C}(\R^2)$. 	
\end{proof}

In the sequel, we provide a nontrivial basis for the bicomplex Hilbert space $L^2_{\BC}(\BC)$. In fact, the Moyal's identity \eqref{Moyal4}
is an effective tool for constructing orthogonal bases for $L^2_{\BC}( \BC )$ from those of $L^2_{\BC}(\C)$.
Namely, we assert

\begin{prop} \label{MainThm3}
	Let $(\phi_{n})_{n}$ be a system in $L^2_{\BC}(\R^2)$ such that $ \phi_n= \phi_n^+ \eu  + \phi_n^- \es $ with $\phi_n^+, \phi_n^- \in L^2_{\C}(\C)$. If $(\phi_n^+)_n$ and $(\phi_n^-)_n$ are orthonormal bases of $L^2_{\C}(\C)$, then the family of functions
	$$ \phi_{m,n}:= \V^{\sigma,\nu,\mu}_{\R^2,\BC} (\phi_{m},\phi_{n} ); \, m,n=0,1,2, \cdots , $$ is an orthonormal basis of $L^{2}_{\BC}( \BC)$.
\end{prop}

\begin{proof}
	Under the assumption that $(\phi_n^+)_n$ and $(\phi_n^-)_n$ are orthogonal in $L^2_{\C}(\C)$, i.e., $$\scal{\phi_n^+,\phi_{'}^+}_{L^2_{\C}(\C)}=\scal{\phi_n^-,\phi_{n'}^-}_{L^2_{\C}(\C)}=0; \quad n\ne n',$$
	it follows
	$$ \scal{\phi_{n},\phi_{n'}}_{L^2_{\BC}(\C)}= \scal{\phi_n^+,\phi_{'}^+}_{L^2_{\C}(\C)} \eu +\scal{\phi_n^-,\phi_{n'}^-}_{L^2_{\C}(\C)} \es = 0$$
	for $n\ne n'$, and therefore $(\phi_{n})_{n}$ is orthogonal in $L^2_{\BC}(\C)$. Thus $$\scal{\phi_{m},\phi_{m'}}_{L^2_{\BC}(\C)}\scal{\phi_{n},\phi_{n'}}_{L^2_{\BC}(\C)}=0 ; \, \mbox{ for } \, (m,n)\ne (m',n').$$
	Subsequently, the family $(\V^{\sigma,\nu,\mu}_{\R^2,\BC} (\phi_{m},\phi_{n} ))_n$ is orthogonal in $L^2_{\BC}(\BC)$ by means of \eqref{Moyal4}.
	Moreover, the corresponding bicomplex norm 	is given by
	\begin{eqnarray}
	 \norm{\phi_{m,n} }_{L^2_{\BC}( \BC)}^2
& =&
	\left|
	\scal{\V^{\sigma,\nu,\mu}_{\R^2,\BC}(\phi_m,\phi_n), \V^{\sigma,\nu,\mu}_{\R^2,\BC}(\phi_m,\phi_n)}_{L^2_{\BC}(\BC)} \right|  \nonumber
\\	&=&\left|  \scal{\phi_m,\phi_m}_{L^2_{\BC}(\C)} \scal{\phi_n,\phi_n}_{L^2_{\BC}(\C)} \right| \nonumber\\
	&=&  \left| \scal{\phi_m^+, \phi_n^+ }_{L^2_{\C}(\C)} \scal{\phi_m^+,\phi_n^+ }_{L^2_{\C}(\C)} \eu
	+ \scal{\phi_m^-,\phi_n^- }_{L^2_{\C}( \C)} \scal{\phi_m^- ,\phi_n^- }_{L^2_{\C}( \C)}\right|  \es  \nonumber 
\\&=&
 \frac1{2}\left(\norm{\phi_m^+}_{L^2_{\C}( \C )}^2 \norm{\phi_n^+}_{L^2_{\C}(\C)}^2 + \norm{\phi_m^-}_{L^2_{\C}( \C)}^2\norm{\phi_n^-}_{L^2_{\C}( \C)}^2\right). \nonumber
	\end{eqnarray}
	so that $ 	\norm{\phi_{m,n} }_{L^2_{\BC}( \BC)}^2=1$ for $(\phi_n^+)_n$ and $(\phi_n^-)_n$ being orthonormal in $L^2_{\C}(\C)$.
The fact that $(\phi_{m,n})_{m,n}$ is a basis of $L^2_{\BC}( \BC) = L^2_{\C}(\C^2)\eu + L^2_{\C}(\C^?2)\es $ follows easily since this is equivalent to $(\V^{\sigma,\nu}_{\R^2,\C^2}(\phi_m^+,\phi_n^+))_{m,n}$ and $(\V^{\sigma,\mu}_{\R^2,\C^2}(\phi_m^-,\phi_n^-))_{m,n}$ be bases of $L^2_{\C}(\C^2)$ in view of the idempotent decomposition \eqref{BidV1}. This holds true since $(\phi_n^+)_n$ and $(\phi_n^-)_n$ are bases of $L^2_{\C}(\C)$ and $\V^{\sigma,\tau}_{\R^2,\C^2}$ is the standard Fourier--Wigner transform mapping orthonormal bases of $L^2_{\C}(\C)$ to orthonormal bases of $L^2_{\C}(\C^2)$. This completes the proof.
\end{proof}

\begin{cor}\label{corBasisLT}
	The functions
	$$h_{m,n,m',n'}^\sigma(Z) := \V^{\sigma,\nu,\mu}_{\R^2,\BC} ( h_{m,n}^\sigma, h_{m',n'}^\sigma )(Z) $$
	for varying $m,n,m',n'=0,1,2, \cdots , $
		form an orthogonal basis of $L^2_{\BC}( \BC)$.
\end{cor}

\begin{proof}
	This is an immediate consequence of Proposition \ref{MainThm3} since the univariate complex Hermite functions $h_{m,n}^\sigma  (\xi,\overline{\xi} ) = h_{m,n}^\sigma  (\xi,\overline{\xi} )\eu + h_{m,n}^\sigma  (\xi,\overline{\xi} ) \es$
    is an orthogonal basis of $L^2_{\BC}( \C)$.
\end{proof}

    \begin{rem}
 The polynomials associated to $h_{m,n,m',n'}^\sigma$ form a new class of bivariate complex Hermite polynomials which are not a tensor product of four one--dimensional copies of the classical Hermite functions $h_n^\sigma$, nor a tensor product of two copies of the complex Hermite functions $h_{m,n}^\sigma$.
  \end{rem}

    \begin{rem}
    	For the special window function $\psi_0(U) := h_{0}^\sigma(u) h_{0}^\sigma(v)=h_{0,0}^\sigma(U_\tau) $ with $U_\tau=u+\tau v$ and $U^2=u^2+v^2$ for $U=(u,v)$, the transfrom $\varphi \longmapsto 	\V^{\sigma,\nu,\mu}_{\R^2,\BC}(\varphi,\psi_0)$ on
    	is closely connected to the bidimensional Segal--Bargmann transform.
    	Indeed, we have
    	\begin{align*}
    	\V^{\sigma,\nu,\mu}_{\R^2,\BC}(\varphi,\psi_0) (Z)
    	&= c \int_{\R^2} e^{\frac{\sigma}2 (\nu\eu+\mu\es) \left( 2UY - XY\right)  - \frac{\sigma}2\left( U-X\right)^2} \varphi(U)d\lambda(U)\\
    		&= c e^{-\frac{\sigma}4 \left(X^2+Y^2\right)}
    		e^{\frac{\sigma}4 [X+ (\nu\eu+\mu\es) Y] ^2}
    		  \int_{\R^2} e^{- \frac{\sigma}2 \left( U- [X+ (\nu\eu+\mu\es) Y]\right)^2}   \varphi(U)d\lambda(U)\\
    & c = e^{-\frac{\sigma}4|S_{\nu\eu+\mu\es}|^2 } e^{\frac{\sigma}4 (S_{\nu\eu+\mu\es})^2 }  \int_{\R^2} e^{- \frac{\sigma}2 \left( U- S_{\nu\eu+\mu\es}\right)^2}  \varphi(U)d\lambda(U),
    	\end{align*}
    	where $c= \sqrt{2\pi}^{-1/2}$ and $S_{\nu\eu+\mu\es}=(z,w)= X+ (\nu\eu+\mu\es) Y \in \C_{\nu\eu+\mu\es}^2$ with $X=(x_1,x_2)$, $Y=(y_1,y_2)$, $Z=z_1+jz_2\in \BC$ and  $z_\ell=x_\ell+iy_\ell$; $\ell=1,2$.
\end{rem}

\section{Concluding remarks}

We have considered two bicomplex analogs of the classical (rescaled) Fourier--Wigner transform. This follows using the idempotent decomposition of bicomplex numbers. The standard phase (or time--frequency) space $\R\times\R$ is replaced here by the bicomplex $(\R\times\R) \eu + (\R\times\R)\es$. Thus the concrete description of analytic properties of these transforms are obtained. It gives rise to special generalization of the bicomplex Bargmann space studied in \cite{Zine2018}. One of the advantage of this setting is to work simultaneously with two models of the polyanalytic Bargmann space $\mathcal{F}^{2,\sigma}_n(\C_\tau)$, the first one 
is focused on $e_+$ and the other 
on $e_-$. This is the case of the first transform and the obtained functional spaces are particular subclasses of the so--called $(n^*,1^{-},1^\dagger)$--$\BC$--polyanalytic functions of first kind. More generally, a bicomplex--valued function $f$ on $\BC$ is said to be $(n^*,m^{-},k^\dagger)$--$\BC$--polyanalytic
if it satisfies  the system of first order differential equations
\begin{eqnarray}
\frac{\partial^{n+1} f}{\partial (Z^*)^{n+1}} =
\frac{\partial^{m+1} f}{\partial \overline{Z}^{m+1}}  =
 \frac{\partial^{k+1} f}{\partial (Z^\dagger)^{k+1}} = 0.
\end{eqnarray}
These spaces (and others) will be the subject of a forthcoming paper.

					As signaled in Section 3, the range of the first transform is strictly contained in $L^2_{\BC}(\BC)$. This is not the case for the second transform studied in Section 4. In fact, we obtain a Hilbertian orthogonal decomposition of $L^2_{\BC}(\BC)$,
					$$L^2_{\BC}(\BC) := \bigoplus_{m,n=0}^{+\infty} \Mbcnumumn  ,$$
	 in terms of the ranges $\Mbcnumumn : = \mathcal{S}^{\sigma,\nu,\mu}_{m,n}(L^2_{\BC}(\R^2))$ of $L^2_{\BC}(\R^2)$ by the transforms $ \mathcal{S}^{\sigma,\nu,\mu}_{m,n} = \V^{\sigma,\nu,\mu}_{\R^2,\BC}( \cdot , h_{m,n}^\sigma)$ (this is contained in Corollary \ref{corBasisLT}).
It will be of interest to provide a concrete description of the functions in Corollary \ref{corBasisLT}. This will be treated in some detail in a forthcoming paper from a different point of view.\\

\noindent{\bf Acknowledgement:} 
The assistance of the members of the "Ahmed Intissar's seminar on Analysis, partial differential equations and spectral geometry" is gratefully acknowledged.


\end{document}